\documentclass{amsart}
\usepackage{amscd,amsmath,amssymb,amsthm,bm,cases,color,comment,mathrsfs,mathtools}
\usepackage{graphicx, color}
\usepackage{tikz}
\usetikzlibrary{decorations.markings}
\usetikzlibrary{arrows.meta}

\usepackage[all]{xy}

\makeatletter
\@namedef{subjclassname@2020}{%
  \textup{2020} Mathematics Subject Classification}
\makeatother

\theoremstyle{definition}
\newtheorem{definition}{Definition}[section]
\newtheorem{remark}[definition]{Remark}

\theoremstyle{plain}
\newtheorem{theorem}[definition]{Theorem}
\newtheorem{proposition}[definition]{Proposition}
\newtheorem{lemma}[definition]{Lemma}
\newtheorem{corollary}[definition]{Corollary}

\numberwithin{equation}{section}

\title[RAAGs and curve graphs of nonorientable surfaces]{Right-angled Artin groups and curve graphs of nonorientable surfaces}

\author[T.~Katayama]{Takuya Katayama}
\address{
(Takuya Katayama)
Department of Mathematics, 
Faculty of Science, 
Gakushuin University, 
1-5-1 Mejiro, Toshima-ku, Tokyo 171-8588, Japan
}
\email{katayama@math.gakushuin.ac.jp}
\author[E.~Kuno]{Erika Kuno}
\address{
(Erika Kuno)
Department of Mathematics,
Graduate School of Science, 
Osaka University,
1-1 Machikaneyama-cho Toyonaka, Osaka 560-0043, Japan
}
\email{e-kuno@math.sci.osaka-u.ac.jp}
\date{\today}

\begin{document}

\begin{abstract}
Let $N$ be a closed nonorientable surface with or without marked points. 
In this paper we prove that, for every finite full subgraph $\Gamma$ of $\mathcal{C}^{\mathrm{two}}(N)$, the right-angled Artin group on $\Gamma$ can be embedded in the mapping class group of $N$. 
Here, $\mathcal{C}^{\mathrm{two}}(N)$ is the subgraph, induced by essential two-sided simple closed curves in $N$, of the ordinal curve graph $\mathcal{C}(N)$. 
In addition, we show that there exists a finite graph $\Gamma$ which is not a full subgraph of $\mathcal{C}^{\mathrm{two}}(N)$ for some $N$, but the right-angled Artin group on $\Gamma$ can be embedded in the mapping class group of $N$. 
\end{abstract}

\maketitle

\section{Introduction}\label{Introduction}

Let $F$ be a closed orientable or nonorientable surface with or without marked points. 
When we are concerned with the orientability and topological type of a surface, by $S_{g, m}$ and $N_{g, m}$ we denote a closed orientable surface of genus $g$ with $m$ marked points and closed nonorientable surface of genus $g$ with $m$ marked points, respectively. 
We also write $S$ and $N$ simply for a closed orientable surface and closed non-orientable surface, respectively. 
Throughout this paper a curve on $F$ means a simple closed curve on $F$. 
A curve $c$ on $F$ is {\it essential} if $c$ bounds neither a disk nor a disk with one marked point, nor a M\"obius band. 
We denote by $\mathcal{C}(F)$ the {\it curve graph} of $F$, which is the simplicial graph that the vertex set consists of the isotopy classes of all essential simple closed curves on $F$ and the edge set consists of all non-ordered pair of essential simple closed curves which can be represented disjointly. 
A curve $c$ is said to be {\it one-sided} if the regular neighborhood of $c$ is a M\"obius band. 
A curve $c$ is said to be {\it two-sided} if the regular neighborhood of $c$ is an annulus. 
For a nonorientable surface $N$, we define the {\it two-sided curve graph} $\mathcal{C}^{\mathrm{two}}(N)$ of $N$, which is the subgraph of $\mathcal{C}(F)$ induced by the isotopy classes of all two-sided curves on $N$. 
Two curves $c_{1}$ and $c_{2}$ in $N$ are in {\it minimal position} if the number of intersections of $c_{1}$ and $c_{2}$ is minimal in the isotopy classes of $c_{1}$ and $c_{2}$. 
We remark that two curves are in minimal position in $N$ if and only if they do not bound a bigon on $N$. 

For a nonorientable surface $N_{g, m}$ with negative Euler characteristic, we define the {\it two-sided complexity} of $N_{g, m}$ as
$$ \xi^{\mathrm{two}}(N)=
    \begin{cases}
        {\frac{3}{2}(g-1)+m-2 \hspace{5mm}(\mathrm{if}~g~\mathrm{is~odd})},\\
        {\frac{3}{2}g+m-3 \hspace{5mm}(\mathrm{if}~g~\mathrm{is~even})},\\      
    \end{cases}
$$
the cardinality of a maximal set of mutually disjoint two-sided simple closed curves in $N_{g, m}$. 
For nonorientable surfaces with non-negative Euler characteristic, $N_{1,0}$, $N_{1,1}$ and $N_{2,0}$, we see that $\xi^{\mathrm{two}}(N_{1,0})=0$, $\xi^{\mathrm{two}}(N_{1,1})=0$, and $\xi^{\mathrm{two}}(N_{2,0})=1$. 
We also define the {\it complexity} $\xi(N)$ of $N$ (resp.\ $\xi(S)$ of $S$) as $\xi(N)={\rm max}\{2g+m-3, 0\}$ (resp. $\xi(S)={\rm max}\{3g+m-3, 0\}$), the cardinality of a maximal set of mutually disjoint simple closed curves in $N$ (resp.\ in $S$). 
Let $\Gamma$ be a finite simplicial graph. 
We denote by ${\it V}(\Gamma)$ and $ E(\Gamma)$ the vertex set and the edge set of $\Gamma$, respectively. 
The {\it right-angled Artin group} on $\Gamma$ is defined by the following group presentation: 
\begin{align*}
A(\Gamma)=\left\langle V(\Gamma)\mid \left[v_{i}, v_{j}\right]=1 {\rm~if~and~only~if}~\{v_{i}, v_{j} \}\in E(\Gamma) \right\rangle.
\end{align*} 
For two groups $G_{1}$ and $G_{2}$, we write $G_{1} \leq G_{2}$ if there exists an injective homomorphism from $G_{1}$ to $G_{2}$. 
Similarly, for two graphs $\Lambda$ and $\Gamma$ we write $\Lambda \leq \Gamma$ if $\Lambda$ is isomorphic to a full subgraph of $\Gamma$. 
The {\it mapping class group} ${\rm Mod}(N)$ of a nonorientable surface $N$ is the set of the isotopy classes of homeomorphisms of $N$ which preserve the set of the marked points. 
The mapping class group $\mathrm{Mod}(S)$ of an orientable surface $S$ is the set of the isotopy classes of orientation-preserving homeomorphisms of $S$ which preserve the set of the marked points. 

Techniques for deciding whether a right-angled Artin group can be embedded in the mapping class group of an orientable surface have been developed during the past decade. 
Koberda~\cite{Koberda12} proved that if $\Gamma\leq \mathcal{C}(S)$, then $A(\Gamma) \leq {\rm Mod}(S)$. 
Koberda's result considerably generalized the classical fact that Dehn twists along a pair of disjoint curves generate a free abelian group of rank two and that the square of Dehn twists along a pair of curves intersecting at least once generate a free group of rank two. 
Recently, Runnels \cite{Runnels21} and Seo \cite{Seo21} independently simplified the Koberda's proof and gave nice bounds for powers of Dehn twists to generate right-angled Artin groups in the mapping class groups of orientable surfaces. 
Kim--Koberda~\cite{Kim-Koberda16} proved that for an orientable surface $S$ with $\xi(S)\geq 4$, there exists a finite graph $\Gamma$ such that $A(\Gamma)\leq {\rm Mod}(S)$ but $\Gamma\not\leq \mathcal{C}(S)$. 
Kim-Koberda proved in \cite[Theorem 5]{Kim-Koberda16} that, for a finite graph with "$n$-thick stars" and an orientable surface with complexity $n$, $A(\Gamma) \leq {\rm Mod}(S)$ implies $\Gamma \leq \mathcal{C}(S)$. 
In \cite{Katayama--Kuno18} the authors proved that, for any linear forest $\Gamma$ and orientable surface $S$, $A(\Gamma^c) \leq {\rm Mod}(S)$ implies $\Gamma^c \leq \mathcal{C}(S)$ (see Section \ref{Preliminaries} for the definition of the ``complement graph" $\Gamma^c$ of $\Gamma$). 

On the other hand, only a little is known about right-angled Artin subgroups of $\mathrm{Mod}(N)$. 
The second author~\cite{Kuno19} computed the maximal rank of the free abelian subgroups of $\mathrm{Mod}(N)$. 
Stukow~\cite{Stukow17} proved that the group generated by two Dehn twists along two-sided curves intersecting at least twice in minimal position is isomorphic to a free group of rank two in $\mathrm{Mod}(N)$. 
However, the authors do not know whether more than two Dehn twists along two-sided curves mutually intersecting at least twice in minimal position generate a free group in $\mathrm{Mod}(N)$. 

\begin{figure}
\begin{tikzpicture}[scale=0.6]
\draw[thick] (0,0) circle [x radius = 3.9cm, y radius =2.5cm];

\draw[thick, shift={(0, 0.3)}] (0,-0.35) circle [x radius=0.65cm, y radius=0.5cm];
\fill[white, shift={(0, 0.3)}] (0,0) circle [radius=0.4cm];
\draw[thick, shift={(0, 0.3)}] (0,0) circle [radius=0.4cm];
\draw[-, thick, shift={(0, 0.3)}] (-0.28, -0.28) -- (0.28, 0.28);
\draw[-, thick, shift={(0, 0.3)}] (-0.28, 0.28) -- (0.28, -0.28);

\draw[thick, shift={(-1.5, 0.3)}] (0,-0.35) circle [x radius=0.65cm, y radius=0.5cm];
\fill[white, shift={(-1.5, 0.3)}] (0,0) circle [radius=0.4cm];
\draw[thick, shift={(-1.5, 0.3)}] (0,0) circle [radius=0.4cm];
\draw[-, thick, shift={(-1.5, 0.3)}] (-0.28, -0.28) -- (0.28, 0.28);
\draw[-, thick, shift={(-1.5, 0.3)}] (-0.28, 0.28) -- (0.28, -0.28);

\draw[thick, shift={(-3.0, 0.3)}] (0,-0.35) circle [x radius=0.65cm, y radius=0.5cm];
\fill[white, shift={(-3.0, 0.3)}] (0,0) circle [radius=0.4cm];
\draw[thick, shift={(-3.0, 0.3)}] (0,0) circle [radius=0.4cm];
\draw[-, thick, shift={(-3.0, 0.3)}] (-0.28, -0.28) -- (0.28, 0.28);
\draw[-, thick, shift={(-3.0, 0.3)}] (-0.28, 0.28) -- (0.28, -0.28);

\draw[thick, shift={(1.5, 0.3)}] (0,-0.35) circle [x radius=0.65cm, y radius=0.5cm];
\fill[white, shift={(1.5, 0.3)}] (0,0) circle [radius=0.4cm];
\draw[thick, shift={(1.5, 0.3)}] (0,0) circle [radius=0.4cm];
\draw[-, thick, shift={(1.5, 0.3)}] (-0.28, -0.28) -- (0.28, 0.28);
\draw[-, thick, shift={(1.5, 0.3)}] (-0.28, 0.28) -- (0.28, -0.28);

\draw[thick, shift={(3.0, 0.3)}] (0,-0.35) circle [x radius=0.65cm, y radius=0.5cm];
\fill[white, shift={(3.0, 0.3)}] (0,0) circle [radius=0.4cm];
\draw[thick, shift={(3.0, 0.3)}] (0,0) circle [radius=0.4cm];
\draw[-, thick, shift={(3.0, 0.3)}] (-0.28, -0.28) -- (0.28, 0.28);
\draw[-, thick, shift={(3.0, 0.3)}] (-0.28, 0.28) -- (0.28, -0.28);

\draw[thick, shift={(8.5, 0)}] (18:2) -- (90:2) -- (162:2) -- (234:2) -- (306:2) --cycle;
\draw[thick, shift={(8.5, 0)}] (18:2) -- (162:2);
\draw[thick, shift={(8.5, 0)}] (234:2) -- (18:2);
\draw[-, line width=8pt, draw=white, shift={(8.5, 0)}] (162:2) to (306:2);
\draw[thick, shift={(8.5, 0)}] (162:2) -- (306:2);
\draw[-, line width=8pt, draw=white, shift={(8.5, 0)}] (90:2) to (234:2);
\draw[thick, shift={(8.5, 0)}] (90:2) -- (234:2);
\draw[-, line width=8pt, draw=white, shift={(8.5, 0)}] (306:2) to (90:2);
\draw[thick, shift={(8.5, 0)}] (306:2) -- (90:2);

\fill[shift={(8.5, 0)}] (18:2) circle [radius=7pt];
\fill[shift={(8.5, 0)}] (90:2) circle [radius=7pt];
\fill[shift={(8.5, 0)}] (162:2) circle [radius=7pt];
\fill[shift={(8.5, 0)}] (234:2)  circle [radius=7pt];
\fill[shift={(8.5, 0)}] (306:2)  circle [radius=7pt];

\end{tikzpicture}
\caption{The complete graph $K_{5}$ on five vertices satisfies $K_5 \leq \mathcal{C}(N_{5,0})$. \label{2}}
\end{figure}
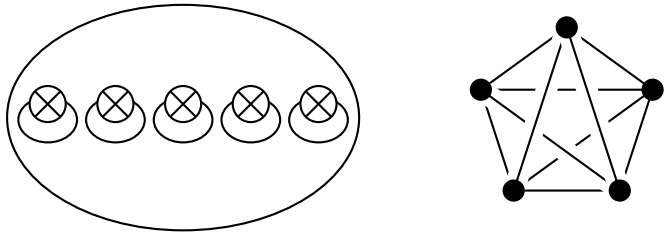

In this paper, we give combinatorial criteria for embedding right-angled Artin groups in $\mathrm{Mod}(N)$. 

\begin{theorem}\label{condition_raag_in_nonorimcg}
For every finite full subgraph $\Gamma \leq \mathcal{C}^{\mathrm{two}}(N)$, we have $A(\Gamma) \leq {\rm Mod}(N)$.
\end{theorem}

From the work \cite{Kuno19} of the second author, $A(K_5) \not\leq {\rm Mod}(N)$; nevertheless $K_5 \leq \mathcal{C}(N)$ (see Figure \ref{2}). 
This shows that the same statement as in Theorem \ref{condition_raag_in_nonorimcg} does not hold for $\mathcal{C}(N)$. 
This phenomenon, however, comes from the fact that the mapping class group of a M\"obius band is trivial. 

In general, $A(\Gamma) \leq {\rm Mod}(N)$ does not imply that $\Gamma \leq \mathcal{C}^{\mathrm{two}}(N)$. 

\begin{theorem}\label{opposite_is_not_true}
Let $N$ be a non-orientable surface homeomorphic to either one of $N_{1,6}$, $N_{3,3}$ and $N_{5,0}$. 
Then there exists a finite graph $\Gamma$ such that $A(\Gamma) \leq \mathrm{Mod}(N)$ but $\Gamma \not\leq \mathcal{C}^{\mathrm{two}}(N)$. 
\end{theorem}

Roughly speaking, a standard embedding of a right-angled Artin group $A(\Gamma)$ in the mapping class group $\mathcal{M}$ is an injective homomorphism from $A(\Gamma)$ into $\mathcal{M}$ such that the image of each vertex is a product of mutually commutative Dehn twists and that the images of two different vertex have different ``supports" (see Definition \ref{standard_embeding}). 
For any closed orientable surface $S$, Kim--Koberda proved that from any injective homomorphism $A(\Gamma) \hookrightarrow \mathrm{Mod}(S)$ we obtain a standard embedding $A(\Gamma) \hookrightarrow \mathrm{Mod}(S)$. 
At the moment the authors do not know whether the same result hold for any closed nonorieantable surface. 
In this paper, assuming there exists a standard embedding, we prove the following. 

\begin{theorem}\label{low_genus}
Let $N$ be a nonorientable surface. 
If $A(\Gamma) \leq \mathrm{Mod}(N)$ and $\xi^{\mathrm{two}}(N) \leq 1$, then $\Gamma \leq \mathcal{C}^{\mathrm{two}}(N)$. 
If there exists a standard embedding $A(\Gamma) \hookrightarrow \mathrm{Mod}(N)$ and $\xi^{\mathrm{two}}(N) = 2$, then $\Gamma \leq \mathcal{C}^{\mathrm{two}}(N)$. 
\end{theorem}

\begin{theorem}\label{n-thick_stars}
Suppose $N$ is a closed nonorientable surfaces with $\xi^{\mathrm{two}}(N)=n$ and $\Gamma$ is a finite graph which has $n$-thick stars. 
If there exists a standard embedding $A(\Gamma) \hookrightarrow \mathrm{Mod}(N)$, then $\Gamma \leq \mathcal{C}^{\mathrm{two}}(N)$. 
\end{theorem}

This paper is organized as follows. 
We first recall terminologies in graph theory and in the theory of the mapping class groups of nonorientable surfaces in Section \ref{Preliminaries}. 
In Section \ref{Proof_of_theorem_condition_raag_in_nonorimcg} we prove Theorem \ref{condition_raag_in_nonorimcg} and introduce some applications of Theorem \ref{condition_raag_in_nonorimcg}  (see Corollaries \ref{LERF} and \ref{bounded_variation}). 
Section \ref{Proof_of_theorem_opposite_is_not_true} is devoted to proving Theorem \ref{opposite_is_not_true}. 
We prove Theorems \ref{low_genus} and \ref{n-thick_stars} in Section \ref{final}. 
The appendix is attached for readers who are interested in injective homomorphisms between right-angled Coxeter groups. 

\section{Preliminaries}\label{Preliminaries}

\subsection{Graph-theoretic terminology}

A {\it graph} is a one-dimensional simplicial complex.
In particular, graphs have neither loops nor multi-edges.
For $X\subseteq V(\Gamma)$, the {\it subgraph} of $\Gamma$ {\it induced by} $X$ is the subgraph $\Lambda$ of $\Gamma$ defined by $V(\Lambda)=X$ and
\begin{align*}
E(\Lambda)=\{e\in E(\Gamma)\mid {\rm the~end~points~of}~e {\rm ~are~in~} X\}.
\end{align*}
In this case, we also say that $\Lambda$ is an {\it induced subgraph} or a {\it full subgraph} of $\Gamma$. 
Throughout this paper, a map between graphs is assumed to map vertices to vertices and edges to edges.
A map $f \colon \Lambda \rightarrow \Gamma$ is said to be {\it full} if $f$ has the following property: for every pair of adjacent vertices $u_1, u_2$ in $\Gamma$ and for any $v_1 \in f^{-1}(u_1)$ and $v_2 \in f^{-1}(u_2)$, the vertices $v_1$ and $v_2$ are adjacent in $\Lambda$. 
A graph $\Gamma$ is $\Lambda$-{\it free} if no full subgraphs of $\Gamma$ are isomorphic to $\Lambda$. 
In particular, $\Gamma$ is {\it triangle}-{\it free} if no full subgraphs of $\Gamma$ are triangles. 
The {\it link} of $v$ in $\Gamma$ is the set of the vertices in $\Gamma$ which are adjacent to $v$, and denoted as ${\rm Link}(v)$. 
The {\it star} of $v$ is the union of ${\rm Link}(v)$ and $\{v\}$, and denoted as ${\rm St}(v)$. 
A {\it clique} on $n$ vertices of a graph $\Gamma$ is a subset of $V(\Gamma)$ with cardinality $n$ which induces a complete subgraph in $\Gamma$. 
By a link, a star, or a clique, we also mean the subgraphs induced by a link, a star, or a clique, respectively. 
The {\it complement graph} $\Gamma^c$ of $\Gamma$ is the graph obtained from the complete graph $K$ on $\# V(\Gamma)$ vertices by deleting the edges corresponding to $E(\Gamma)$ so that $E(\Gamma) \sqcup E(\Gamma^c) = E(K)$. 
For a positive integer $n$, we say that $\Gamma$ has $n$-{\it thick stars} if each vertex $v$ of $\Gamma$ is contained in a pair of distinct cliques on $n$ vertices of $\Gamma$ whose intersection is exactly $\{ v \}$. 
Obviously, a graph $\Gamma$ has $n$-thick stars if and only if ${\rm Link}(v)$ contains two disjoint copies of complete graphs on $n-1$ vertices of $\Gamma$ for each vertex $v$. 
We can find many examples for graphs with $n$-thick stars in the $1$-skeleta of triangulations of $n$-dimensional hyperbolic manifolds. 

\subsection{Nonorientable surfaces}

Let $N=N_{g, m}$ be a closed nonorientable surface of genus $g$ with $m$ marked points. 
Note that $N$ is homeomorphic to the surface obtained from a sphere with $m$ marked points by removing $g$ open disks and attaching $g$ M\"{o}bius bands along their boundaries, and each of the attached M\"{o}bius bands is called a {\it crosscap} (see Figure \ref{fig_two_pattern_nonori_surface}).

\begin{figure}[h]
\includegraphics[scale=0.63]{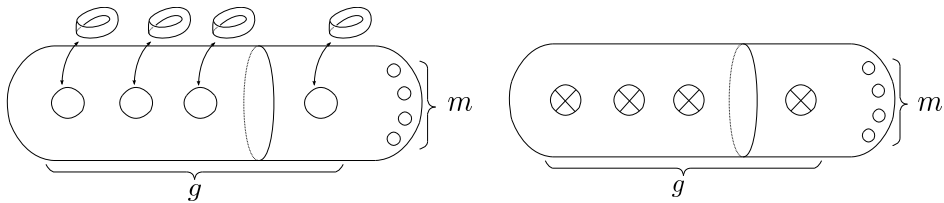}
\caption{A nonorientable surface $N_{g,m}$. \label{fig_two_pattern_nonori_surface}}
\end{figure}

For a two-sided curve $c$ on $N$, we choose an orientation of the regular neighborhood of $c$ in $N$. 
Then a {\it Dehn twist} $t_{c}$ is a homeomorphism of $N$ defined by a combination of the following operations: cutting $N$ along $c$, twisting one side to the right with respect to the orientation by $2\pi$, and reglueing (see Figure \ref{fig_dehn_twist}). 
We also denote the isotopy class of a Dehn twist along $c$ by $t_{c}$.

\begin{figure}[h]
\includegraphics[scale=0.50]{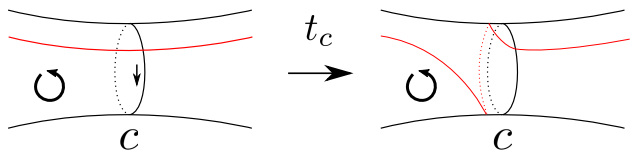}
\caption{A Dehn twist $t_{c}$ along a two-sided curve $c$ on $N_{g, m}$. \label{fig_dehn_twist}}\label{fig_dehn_twist}
\end{figure}

A {\it two-sided multi-curve} is a set of two-sided curves on $N$ which are in minimal position, mutually disjoint and non-isotopic.

\begin{figure}[h]
\includegraphics[scale=0.40]{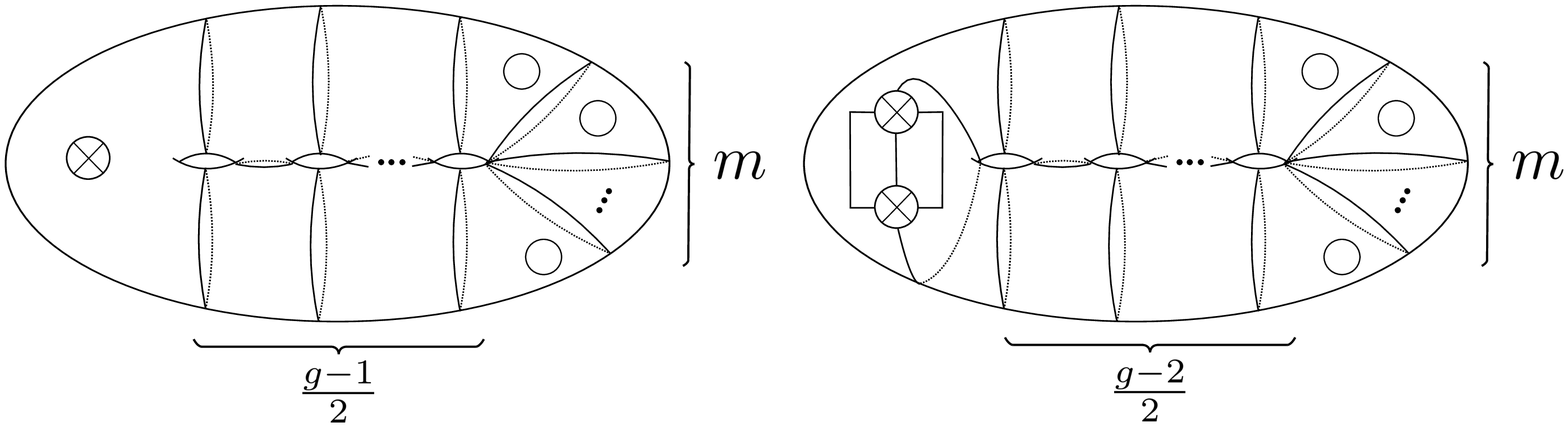}
\caption{The maximal number of pairwise disjoint two-sided curves on surfaces of odd genera (left-hand side) and surfaces of even genera (right-hand side). \label{fig_two_sided_scc}}
\end{figure}

We note that the two-sided complexity $\xi^{\mathrm{two}}(N)$ is the maximal number of the components of the two-sided multi-curves on $N$ (c.f.\ Atalan--Korkmaz~\cite{Atalan-Korkmaz14} and Kuno~\cite{Kuno19}). 
A {\it multi-twist} on $N$ is a homeomorphism of $N$ defined by the composition of right-handed and left-handed Dehn twists along components of a two-sided multi-curve. 

Let $S$ be the orientation double cover of a closed nonorientable surface $N$ with negative Euler characteristic. 
Then there exists an injective homomorphism $\iota\colon\mathrm{Mod}(N)\hookrightarrow\mathrm{Mod}(S)$. 
Precisely, we have the following. 

\begin{lemma}[\cite{Goncalves--Guaschi--Maldonado18}] \label{ori_cov_mcg_inj}
Let $N$ be a nonorientable surface with negative Euler characteristic and $S$ be the orientation double cover of $N$ with the covering transformation $J$. 
Then $\mathrm{Mod}(N) \leq Z^{\pm}([J])$ and the image consists of the mapping classes of all orientation-preserving homeomorphisms commutative with $J$. 
Here, $Z^{\pm}([J])$ is the centralizer of the mapping class $[J]$ in $\mathrm{Mod}^{\pm}(S)$ and $\mathrm{Mod}^{\pm}(S)$ is the group of homeomorphisms of $S$ preserving the set of the marked points, up to isotopy relative to the set of the marked points. 
\end{lemma}


\section{Proof and applications of Theorem \ref{condition_raag_in_nonorimcg}}\label{Proof_of_theorem_condition_raag_in_nonorimcg}

Our proof for Theorem \ref{condition_raag_in_nonorimcg} will use a combinatorial method for embedding right-angled Artin groups into other right-angled Artin groups. 
For this purpose we recall a useful solution for the word problem of right-angled Artin groups. 
Suppose that $\Gamma$ is a graph. 
An element of $V(\Gamma) \cup V(\Gamma)^{-1}$ is called a {\it letter}. 
Any element in $A(\Gamma)$ can be expressed as a {\it word}, which is a finite multiplication of letters. 
Let $w= a_1 \cdots a_l$ be a word in $A(\Gamma)$, where $a_1, \ldots, a_l$ are letters. 
We say that $w$ is {\it reduced} if any other word representing the same element as $w$ in $A(\Gamma)$ has at least $l$ letters. 

\begin{lemma}[{cf.\ \cite[Section 5]{Crisp--Wiest04}}]
\label{reduced}
Let $w$ be a word in $A(\Gamma)$. 
Then $w$ is reduced if and only if $w$ does not contain a word of the form $v^{\varepsilon}xv^{- \varepsilon}$, where $v$ is a vertex of $\Gamma$, $\varepsilon \in \{1, -1 \}$ and $x$ is a word such that $v$ is commutative with all of the letters in $x$. 
\end{lemma}

An important observation is that a surjective full map with a bit condition between graphs induces an embedding of a right-angled Artin group.  

\begin{lemma} \label{co_hom_inj}
Let $f \colon \Lambda \rightarrow \Gamma$ be a surjective full map. 
Suppose that $f$ satisfies the following condition $(*)$: for each pair of non-adjacent vertices $u \neq u'$ in $\Gamma$ and for any vertex $v$ in $f^{-1}(u)$, there exists a vertex $v'$ in $f^{-1}(u')$ such that $v$ is not adjacent to $v'$. 
Then the ``diagonal" map $\phi \colon V(\Gamma) \rightarrow A(\Lambda)$ sending each vertex $u$ to the product $\coprod_{v \in f^{-1}(u)} v$ extends to an injective group homomorphism $\varphi \colon A(\Gamma) \hookrightarrow A(\Lambda)$. 
Moreover, $\varphi$ is a quasi-isometric embedding. 
\end{lemma}
\begin{proof}[Proof of Lemma \ref{co_hom_inj}]
Note that the correspondence $\phi$ is a map because $f$ is surjective. 
We first prove that $\phi$ extends to a group homomorphism $\varphi$. 
Pick vertices $u_1$ and $u_2$ in $\Gamma$ with the relation $[u_1, u_2]=1$. 
Since $f$ is full, every vertex in $f^{-1}(u_1)$ is commutative with all vertices in $f^{-1}(u_2)$. 
This implies that $\phi(u_1)$ and $\phi(u_2)$ are commutative. 
Hence, $\phi$ extends to a group homomorphism $\varphi$. 

We next prove that $\varphi$ is injective. 
Pick an element $g$ of $\mathrm{Ker}\varphi$ and choose a reduced representation $w$ of $g$. 
Suppose, on the contrary, that $w$ is not empty. 
Since $\varphi$ is a group homomorphism, we can regard $\varphi(w)$ as a word representation of $\varphi(g)$. 
Since $w$ is not empty, $\varphi(w)$ is not empty either. 
However, since a reduced representation of $\varphi(g)$ is empty, according to Lemma \ref{reduced} a cancelation must occur in $\varphi(w)$ and the cancelation is of the form $v^{\varepsilon} x v^{- \varepsilon}$ such that $v$ is a vertex of $\Lambda$ and $x$ is a word in $A(\Lambda)$ consisting of vertices commutative with $v$, and $\varepsilon \in \{ 1, -1 \}$. 
Since $f$ is a map, we have $f^{-1}(u_1) \cap f^{-1}(u_2) = \emptyset$ for every pair of distinct vertices $u_1$ and $u_2$ in $\Gamma$. 
Hence, the prefix and suffix of $v^{\varepsilon} x v^{- \varepsilon}$ come from the identical vertex $u_{v}$ in $\Gamma$, namely, $v \in f^{-1}(u_{v})$. 
This implies that $w$ contains a subword of the form $u_v^{\varepsilon} \bar{x} u_v^{- \varepsilon}$, where $\bar{x}$ is a word in $A(\Gamma)$ and $\varphi(\bar{x})$ is a subword of $x$. 
Since $\varphi(\bar{x})$ is a subword of $x$ and since $f$ satisfies the property $(*)$, every letter  contained in $\bar{x}$ is adjacent to $u_v$ in $\Lambda$. 
This implies that a cancelation $u_v^{\varepsilon} \bar{x} u_v^{- \varepsilon} \equiv \bar{x}$ occurs in $w$, and therefore $w$ is not reduced, a contradiction. 

By the same argument, we can verify that $\varphi$ is a bi-Lipschitz map with a Lipschitz constant $\mathrm{max} \{ \# f^{-1}(v) \mid v \in V(\Lambda) \}$, because reduced words are shortest representations in right-angled Artin groups and $\varphi$ maps any reduced word to a reduced word. 
\end{proof}

Let $\pi \colon S \rightarrow N$ be the orientation double covering. 
Given a curve $\alpha$ on $N$, we have the set of lifts $\tilde{\alpha} = \{ \gamma, J(\gamma) \}$ on $S$ (a curve $J(\gamma)$ is possible to be $\gamma$). 
From a finite full subgraph $\Gamma \leq \mathcal{C}^{\mathrm{two}}(N)$, we obtain a finite subgraph $\tilde{\Gamma} \leq \mathcal{C}(S)$ induced by a union of the sets of lifts. 

\begin{lemma} \label{cov_raag_inj}
Let $\Gamma$ be a finite full subgraph of $\mathcal{C}^{\mathrm{two}}(N)$. 
Then there exists a surjective full map $\tilde{\Gamma} \rightarrow \Gamma$ satisfying condition $(*)$ in Lemma \ref{co_hom_inj}. 
\end{lemma}
\begin{proof}[Proof of \ref{cov_raag_inj}]
The covering map $\pi \colon S \rightarrow N$ induces a surjection of $\tilde{\Gamma}$ onto $\Gamma$, denoted  by $p$. 
Since a pair of disjoint curves is lifted to a system of mutually disjoint curves, the map $p$ is full.  
To see that $p$ satisfies condition $(*)$, suppose that $\alpha$ and $\beta$ are curves on $N$ intersecting at  least once. 
Pick a connected component $c$ of a union of the sets of lifts, $\tilde{\alpha}$ and $\tilde{\beta}$, which contains lifts of $\alpha$ and $\beta$. 
If $c$ forms a bigon, then $\pi(c) = \alpha \cup \beta$ also forms a bigon. 
Hence, if $\alpha$ and $\beta$ are in minimal position on $N$, then each pair of curves in $\tilde{\alpha} \cup \tilde{\beta}$ are in minimal position on $S$. 
This implies that $p$ satisfies condition $(*)$ in Lemma \ref{co_hom_inj}. 
\end{proof}

\begin{proof}[Proof of Theorem \ref{condition_raag_in_nonorimcg}]
Suppose that $\Gamma$ is a finite full subgraph of $\mathcal{C}^{\mathrm{two}}(N)$. 
By Lemmas \ref{co_hom_inj} and \ref{cov_raag_inj}, we have an embedding $\varphi \colon A(\Gamma) \hookrightarrow A(\tilde{\Gamma})$. 
Since the vertex set of $\tilde{\Gamma}$ corresponds to a curve system on $S$, there exists a number $n$ such that the $n$-th powers of Dehn twists induce an embedding $\psi \colon A(\tilde{\Gamma}) \hookrightarrow \mathrm{Mod}(S)$ by Koberda's theorem \cite[Theorem 1.1]{Koberda12}. 
Hence, we have an embedding $\psi \circ \varphi \colon A(\Gamma) \hookrightarrow  \mathrm{Mod}(S)$. 
By the definition of $\tilde{\Gamma}$, the image of $\psi \circ \varphi$ consists of mapping classes  represented by homeomorphisms commutative with the deck transformation $J$. 
This implies that $A(\Gamma) \leq Z^{\pm}([J])$, in particular $A(\Gamma) \leq \mathrm{Mod}(N)$. 
\end{proof}

Runnels in \cite{Runnels21} and Seo in \cite{Seo21} gave bounds for generating right-angled Artin groups in the mapping class groups of orientable surfaces in terms of intersection number. 
Our proof for Theorem \ref{condition_raag_in_nonorimcg} together with their results also implies that we can give bounds for generating right-angled Artin groups in the mapping class groups of nonorientable surfaces in terms of intersection number. 

Let us introduce some applications of Theorem \ref{condition_raag_in_nonorimcg}. 
Let $P_4$ be the path graph on four vertices (the graph homeomorphic to the closed unit interval, which consists of four vertices). 

\begin{proposition} \label{p4_nonori}
The right-angled Artin group $A(P_4)$ can be embedded in $\mathrm{Mod}(N_{1, p})$ if and only if $p \geq 4$. 
\end{proposition}
\begin{proof}[Proof of Proposition \ref{p4_nonori}]
Suppose $p \geq 4$. 
We can find a sequence of two-sided essential simple curves $(\alpha_1, \alpha_2, \alpha_3, \alpha_4)$ such that $\alpha_i$ and $\alpha_{i+1}$ intersect exactly twice in minimal position and non-adjacent curves are represented disjointly on $N_{1, p}$. 
Hence, $P_4 \leq \mathcal{C}^{two}(N_{1, p})$. 
Applying Theorem \ref{condition_raag_in_nonorimcg} to this finite full subgraph, we have $A(P_4) \leq \mathrm{Mod}(N_{1, p})$. 

Suppose $p \leq 3$. 
Then $\mathrm{Mod}(N_{1, p})$ does not contain $\mathbb{Z}^{2}$, because the two-sided complexity of $N_{1, p}$ is at most one. 
Since $A(P_4)$ contains a subgroup isomorphic to $\mathbb{Z}^{2}$, $A(P_4)$ cannot be embedded in $\mathrm{Mod}(N_{1, p})$. 
\end{proof}

A group $G$ is said to be {\it locally extended residually finite} if for any finitely generated subgroup $H$ of $G$ and any element $g \in G - H$, there exists a finite quotient $\phi \colon G \rightarrow Q$ such that $\phi(g)$ is non-trivial in $Q$. 
In other words, a locally extended residually finite group is a group whose finitely generated subgroups are closed subgroups in the profinite topology. 
As proved in \cite[Theorem 1.2]{Niblo--Wise00}, the right-angled Artin group $A(P_4)$ is not locally extended residually finite. 

\begin{corollary} \label{LERF}
Suppose $p \neq 3$. 
The mapping class group $\mathrm{Mod}(N_{1, p})$ is locally extended residually finite if and only if $p \leq 2$. 
\end{corollary}
\begin{proof}[Proof of \ref{LERF}]
Suppose $p \leq 2$. 
Since $\mathrm{Mod}(N_{1, p})$ is a finite group, $\mathrm{Mod}(N_{1, p})$ is locally extended residually finite. 
For the isomorphism type of $\mathrm{Mod}(N_{1, p})$, see \cite[Theorem 4.1 and Corollary 4.6]{Korkmaz02}. 

Suppose $p \geq 4$. 
By Proposition \ref{p4_nonori}, $A(P_4) \leq \mathrm{Mod}(N_{1, p})$. 
Since $A(P_4)$ is not locally extended residually finite, $\mathrm{Mod}(N_{1, p})$ is not locally extended residually finite either. 
\end{proof}

It is very natural to think that $\mathrm{Mod}(N_{1, 3})$ is locally extended residually finite, but the authors do not have a proof. 
As is well-known, the mapping class groups of orientable or nonorientable surfaces are residually finite, i.e., for every non-trivial element $g$ in the mapping class group, there exists a finite quotient that realizes $g$ as a non-trivial element. 

Baik--Kim--Koberda in \cite{Baik--Kim--Koberda} studied the ``smoothability" of actions of mapping class groups to compact one-manifolds. 
By combining their results and Proposition \ref{p4_nonori}, we have the following. 

\begin{corollary} \label{bounded_variation}
Let $M$ be a compact one-manifold and $p$ be an integer $\geq 4$. 
Then no finite index subgroup of $\mathrm{Mod}(N_{1, p})$ is embedded in the group of orientation-preserving $C^1$-diffeomorphisms of $M$ whose first derivative have bounded variation. 
\end{corollary}

This corollary implies that no finite index subgroup of $\mathrm{Mod}(N_{1, p})$ is embedded in the group of orientation-preserving $C^2$-diffeomorphisms of a compact one-manifold if $p \geq 4$. 

\begin{remark}
Suppose $p \geq 1$. 
By Lemma \ref{ori_cov_mcg_inj}, $\mathrm{Mod}(N_{1, p})$ is embedded in $\mathrm{Mod}(S_{0, 2p})$. 
According to Farb--Franks theorem (see \cite[Theorem 1.9]{Baik--Kim--Koberda}), $\mathrm{Mod}(S_{0, 2p})$ virtually embeds into the group $\mathrm{Diff}_{+}^{1}(M)$ of orientation-preserving $C^1$-diffeomorphism of a connected one-manifold $M$. 
Hence, $\mathrm{Mod}(N_{1, p})$ also virtually embeds into $\mathrm{Diff}_{+}^{1}(M)$. 
\end{remark}

See also a recent work \cite{Kim--Koberda--Rivas} due to Kim--Koberda--Rivas, for more applications of embedding right-angled Artin groups into groups. 

\begin{remark} \label{kapovich_rem}
Note that Lemma \ref{co_hom_inj} generalizes \cite[Lemma 2.3]{Kapovich}. 
In \cite{Kapovich}, Kapovich proved that any surjective and locally surjective graph morphism $\Lambda \rightarrow \Gamma$ between finite graphs induces the diagonal embedding $A(\Gamma^c) \hookrightarrow A(\Lambda^c)$ similarly as in the statement of Lemma \ref{co_hom_inj}. 
Here, a map $f$ between graphs is called a {\it graph morphism} if $f$ sends each pair of adjacent vertices to adjacent vertices.  
A map $f \colon \Lambda \rightarrow \Gamma$ is said to be {\it locally surjective} if for every vertex $v$ in $\Lambda$ and every edge $e$ of $\Gamma$ incident to $f(v)$, there exists an edge $\tilde{e}$ of $\Lambda$ such that $\tilde{e}$ is incident to $v$ and that $f(\tilde{e}) = e$. 
Our definition of local surjectivity is not equivalent to the definition in \cite{Kapovich}, but we have to employ our definition to prove \cite[Lemma 2.3]{Kapovich}. 

Given a map $f \colon \Lambda \rightarrow \Gamma$, we can consider the $0$-dimensional map $f^{0} \colon V(\Lambda) \rightarrow V(\Gamma); u \mapsto f(u)$. 
We regard the map $f^{0}$ as a map from $\Lambda^c$ to $\Gamma^c$ which is called the {\it complementary map} of $f$ and denoted by $f^c$. 
The complementary map $f^c$ satisfies condition $(*)$ if and only if $f$ is locally surjective. 
If $f$ is a graph morphism, then $f^c$ is full. 
This implies that Lemma \ref{co_hom_inj} is a generalization of Kapovich's result. 
We note that, if every fiber of a surjective full map $f^c \colon \Lambda^c \rightarrow \Gamma^c$ with condition $(*)$ consists of mutually adjacent vertices in $\Lambda^c$, then we have a surjective and locally surjective graph morphism $\Lambda \rightarrow \Gamma$. 
\end{remark}

\begin{remark}
Let $p \colon \mathcal{T} \rightarrow \Gamma$ be the universal covering of a graph $\Gamma$. 
Kim--Koberda in \cite{Kim--Koberda15} proved that there exists a finite subtree $T \leq \mathcal{T}$ such that the diagonal homomorphism $A(\Gamma^c) \hookrightarrow A(T^c)$ is injective. 
Let us explain the difference between Kim--Koberda's method and Lemma \ref{co_hom_inj}. 
The diagonal homomorphism obtained in Lemma \ref{co_hom_inj} has a property that, for any pair of non-adjacent vertices $v_1 \neq v_2$, the word given by $\varphi(v_2^{-1}) \varphi(v_1) \varphi(v_2)$ is reduced in the first place. 
On the other hand, the diagonal homomorphism obtained from Kim--Koberda's method does not have such a property in general. 
Moreover, a fiber of a surjective full map with condition $(*)$ can form a finite graph arbitrarily. 
On the other hand, each fiber of the complementary map of a covering map forms a clique. 
\end{remark}


\section{Proof of Theorem~\ref{opposite_is_not_true}}\label{Proof_of_theorem_opposite_is_not_true}

Let $\Gamma_{0}$ and $\Gamma_{1}$ be the finite graphs shown in Figure~\ref{fig_Gamma_0_and_Gamma_1}.
We denote by $C_{4}$ the 4-cycle spanned by $\{a, b, c, d\}$. 
Let $\varphi \colon A(\Gamma_{0})\rightarrow A(\Gamma_{1})$ be the map defined by $\varphi(q)=ef$ and $\varphi(v)=v$ for any $v\in V(\Gamma_{0})-\{q\}$. 
For a graph $\Gamma$, we will denote by $\langle v\rangle$ the subgroup of $A(\Gamma)$ generated by a single vertex $v\in V(\Gamma)$. 

\begin{figure}[h]
\begin{center}
\includegraphics[scale=0.35]{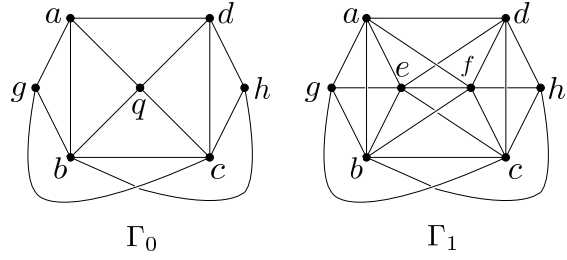}
\caption{Two graphs $\Gamma_{0}$ and $\Gamma_{1}$.}\label{fig_Gamma_0_and_Gamma_1}
\end{center}
\end{figure}

\begin{lemma}[{\cite[Lemma 11]{Kim-Koberda16}}]\label{first_lemma_of_third_thm_Kim-Koberda16}
The homomorphism $\varphi \colon A(\Gamma_{0})\rightarrow A(\Gamma_{1})$ is injective.
\end{lemma}

\begin{proof}
See \cite[Proof of Lemma 11]{Kim-Koberda16}. 
\end{proof}

\begin{lemma}\label{embedding_Gamma_{1}_into_two-sided_curve_graphs}
For $N=N_{1, 6}$, $N=N_{3, 3}$, or $N=N_{5, 0}$, the graph $\Gamma_{1}$ is embedded into $\mathcal{C}^{\mathrm{two}}(N)$
\end{lemma}

\begin{figure}[h]
\begin{center}
\includegraphics[scale=0.693]{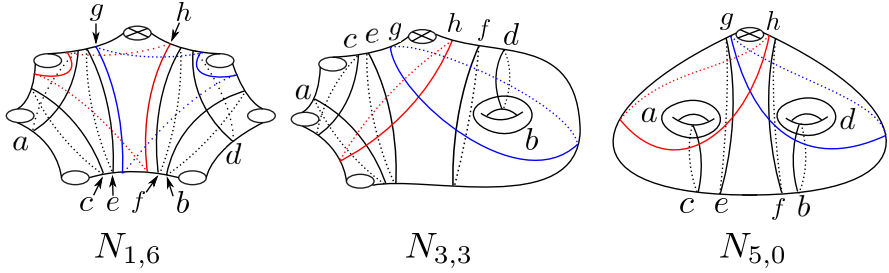}
\caption{Realizations for $\Gamma_1$ on nonorientable surfaces $N_{1, 6}$, $N_{3, 3}$, and $N_{5, 0}$.}\label{fig_three_nonori_surfaces_v1}
\end{center}
\end{figure}

\begin{proof}[Proof of Lemma~\ref{embedding_Gamma_{1}_into_two-sided_curve_graphs}]
We show $\Gamma_{1}$ is embedded into $\mathcal{C}^{\mathrm{two}}(N_{1, 6})$, $\mathcal{C}^{\mathrm{two}}(N_{3, 3})$, and $\mathcal{C}^{\mathrm{two}}(N_{5, 0})$ as a full subgraph. 
Note that the two-sided complexities of $N_{1, 6}$, $N_{3, 3}$, and $N_{5, 0}$ are all four. 
We put two-sided curves $a$, $b$, $c$, $d$, $e$, $f$, $g$, and $h$ in the nonorientable surfaces as in Figure~\ref{fig_three_nonori_surfaces_v1} so that they form the graph $\Gamma_{1}$ in the two-sided curve graphs. 
One can verify that the curves are in minimal position, since they do not bound a bigon in the surfaces. 
Hence $\Gamma_{1}$ is embedded into $\mathcal{C}^{\mathrm{two}}(N_{1, 6})$, $\mathcal{C}^{\mathrm{two}}(N_{3, 3})$, and $\mathcal{C}^{\mathrm{two}}(N_{5, 0})$ as a full subgraph. 
\end{proof}

\begin{remark} \label{the_embeddability_of_gamma_1_in_two-sided_curve_graphs}
The graph $\Gamma_{1}$ is embedded into $\mathcal{C}^{\mathrm{two}}(N)$ if and only if $N$ is a nonorientable surface whose two-sided complexity is at least $5$, or whose two-sided complexity is $4$ with $N=N_{1, 6}, N_{3, 3}, N_{5, 0}$.
\end{remark}

By combining Theorem~\ref{condition_raag_in_nonorimcg}, Lemmas~\ref{first_lemma_of_third_thm_Kim-Koberda16} and~\ref{embedding_Gamma_{1}_into_two-sided_curve_graphs}, we have the following corollary: 

\begin{corollary}\label{possibility_to_be_subgroup}
For $N=N_{1, 6},~N_{3, 3},~N_{5, 0}$, the right-angled Artin group $A(\Gamma_{0})$ is embedded in the mapping class group $\mathrm{Mod}(N)$. 
\end{corollary}

From now on we prove the rest of Theorem~\ref{opposite_is_not_true}, that is, for $N=N_{1, 6},~N_{3, 3},~N_{5, 0}$ the graph $\Gamma_{0}$ cannot be embedded in $\mathcal{C}^{\mathrm{two}}(N)$. 

Let $N$ be a closed nonorientable surface with $\xi^{\mathrm{two}}(N)=4$.
Suppose \{$a$, $b$, $c$, $d$\} are two-sided curves on $N$ which form a four cycle $C_{4}$ in $\mathcal{C}^{\mathrm{two}}(N)$ with this order.
Let $F_{1}$ be the regular neighborhood of $a$ and $c$ in $N$, and $F_{2}$ the regular neighborhood of $b$ and $d$ in $N$ so that $F_{1}\cap F_{2}=\emptyset$. 
Set $F_{0}=\overline{\partial N-(F_{1}\cup F_{2})}$. 
We remark that $F_{0}$, $F_{1}$, and $F_{2}$ can be orientable surfaces. 
We regard the boundaries of $F_{0}$, $F_{1}$, and $F_{2}$ as marked points from now. 

\begin{lemma}\label{six_cases_of_the_two-sided_complexity_four_nonorientable_surfaces}
The triple $(F_{0}, F_{1}, F_{2})$ satisfies exactly one of the following six cases, possibly after switching the roles of $F_{1}$ and $F_{2}$, where the following $S^{1}$ means a circle in $N$.

\begin{itemize}
\item[{\rm (1)}] $F_{0}\approx S_{0,2}$, and $F_{0}\cap F_{1}\approx S^{1}$ and $F_{0}\cap F_{2}\approx S^{1}$, and we choose appropriate pairs $(F_{1}, F_{2})$ from $F_{1}\in\{N_{1,3}, N_{2,1}, S_{0,4}, S_{1,1}\}$, $F_{2}\in\{N_{1,4}, N_{2,2}, S_{0,5}, S_{1,2}\}$.

\item[{\rm (2)}] $F_{0}\in\{N_{1,2}, S_{0,3}\}$,  and $F_{0}\cap F_{1}\approx S^{1}$ and $F_{0}\cap F_{2}\approx S^{1}$, and we choose appropriate pairs $(F_{1}, F_{2})$ from $F_{1}, F_{2}\in\{N_{1,3}, N_{2,1}, S_{0,4}, S_{1,1}\}$.

\item[{\rm (3)}] $F_{0}\approx S_{0,2}\coprod S_{0,2}$, and $F_{0}\cap F_{1}\approx S^{1}\coprod S^{1}$ and $F_{0}\cap F_{2}\approx S^{1}$, and we choose appropriate pairs $(F_{1}, F_{2})$ from $F_{1}, F_{2}\in\{N_{1,3}, S_{0,4}\}$.

\item[{\rm (4)}] $F_{0}\approx S_{0,2}\coprod S_{0,2}$, and $F_{0}\cap F_{1}\approx S^{1}\coprod S^{1}\coprod S^{1}$ and $F_{0}\cap F_{2}\approx S^{1}$, and we choose appropriate pairs $(F_{1}, F_{2})$ from $F_{1}\in\{N_{1,3}, S_{0,4}\}$, $F_{2}\in\{N_{1,3}, N_{2,1}, S_{0,4}, S_{1,1}\}$.

\item[{\rm (5)}] $F_{0}\approx S_{0,2}\coprod S_{0,3}$, and $S_{0,2}\cap F_{1}\approx S^{1}$, $S_{0,2}\cap F_{2}\approx S^{1}$, $S_{0,3}\cap F_{1}\approx S^{1}$, $S_{0,3}\cap F_{2}=\emptyset$, and we choose appropriate pairs $(F_{1}, F_{2})$ from $F_{1}\in\{N_{1,3}, S_{0,4}\}$, $F_{2}\in\{N_{1,3}, N_{2,1}, S_{0,4}, S_{1,1}\}$.

\item[{\rm (6)}] $F_{0}\approx S_{0,2}\coprod N_{1,2}$, and $S_{0,2}\cap F_{1}\approx S^{1}$, $S_{0,2}\cap F_{2}\approx S^{1}$, $N_{1,2}\cap F_{1}\approx S^{1}$, $N_{1,2}\cap F_{2}=\emptyset$, and we choose appropriate pairs $(F_{1}, F_{2})$ from $F_{1}\in\{N_{1,3}, S_{0,4}\}$, $F_{2}\in\{N_{1,3}, N_{2,1}, S_{0,4}, S_{1,1}\}$.
\end{itemize}
\end{lemma}

\begin{proof}[Proof of Lemma~\ref{six_cases_of_the_two-sided_complexity_four_nonorientable_surfaces}]
Let $N$ be a compact connected nonorientable surface with $\xi^{\mathrm{two}}(N)=4$. 
We define $\alpha$ to be the number of the free isotopy classes of the boundary components of $F_{0}$ that are contained in $F_{1}\cup F_{2}$. 
We have $\alpha>0$, since $N$ is connected and $F_{1}\cap F_{2}=\emptyset$. 
For $i=0,1,2$, we set
$$ \zeta(F_{i})\coloneqq
    \begin{cases}
        {\xi^{\mathrm{two}}(F_{i}) \hspace{5mm}(\mathrm{if}~F_{i}~\mathrm{is~a~nonorientable~surface})},\\
        {\xi(F_{i}) \hspace{5mm}(\mathrm{if}~F_{i}~\mathrm{is~a~orientable~surface})}.\\      
    \end{cases}
$$
Then $\xi^{\mathrm{two}}(N)=\zeta(F_{1})+\zeta(F_{2})+\zeta(F_{0})+\alpha$. 
Since $F_{i}$ ($i=1, 2$) contains at least one two-sided curve, we have $2\leq \zeta(F_{1})+\zeta(F_{2})$.
Moreover, $\zeta(F_{1})+\zeta(F_{2})=4-\zeta(F_{0})-\alpha\leq 4-0-1=3$.
Therefore it follows that $2\leq \zeta(F_{1})+\zeta(F_{2})\leq 4$.

Suppose that $\zeta(F_{1})+\zeta(F_{2})=3$. 
Without loss of generality we may assume $\zeta(F_{1})=1$ and $\zeta(F_{2})=2$. 
Then we have $\zeta(F_{0})+\alpha=1$. 
If $\zeta(F_{1})=1$ and $\zeta(F_{2})=2$, then $F_{1}\in\{N_{1,3}, N_{2,1}, S_{0, 4}, S_{1,1}\}$ and $F_{2}\in\{N_{1,4}, N_{2,2}, S_{0, 5}, S_{1, 2} \}$. 
By the assumption that $\alpha\geq 1$, we have $\alpha=1$ and $\xi(F_{0})=0$. 
Since $F_{0}$ has at least two boundary components, $F_{0}\in\{N_{1,2}, S_{0, 2}, S_{0, 3} \}$. 
If $F_{0}\approx N_{1,2}$ or $S_{0, 3}$, then this contradicts the assumption that $\alpha=1$. 
Hence, we have $F_{0}\approx S_{0, 2}$. 
Case (1) is immediate. 

We suppose that $\zeta(F_{1})+\zeta(F_{2})=2$.
Then we have $\zeta(F_{0})+\alpha=2$ and $\zeta(F_{1})=\zeta(F_{2})=1$.
If $\alpha=1$, then $\zeta(F_{0})=1$.
If $F_{0}$ is connected, then $F_{0}\in\{N_{1,2}, S_{0, 2}, S_{0, 3} \}$.
Since $\alpha=1$, it follows that $F_{0}\approx S_{0,2}$, but this contradicts the fact that $\zeta(F_{0})=1$. 
If $F_{0}$ is not connected, then it also contradicts the fact that $\alpha=1$. 

So we have $\alpha=2$ and $\zeta(F_{0})=0$. 
If $F_{0}$ is connected, then $\alpha=2$ implies that $F_{0}\approx N_{1,2}$ or $S_{0, 3}$, and $F_{0}$ intersects both $F_{1}$ and $F_{2}$. 
By the assumption that $\zeta(F_{1})=\zeta(F_{2})=1$, we have $F_{1}, F_{2} \in \{ N_{1,3}, N_{2,1}, S_{0,4}, S_{1,1} \}$. 
Hence Case (2) follows. 

We assume that $F_{0}$ is not connected.
By the assumption that $\alpha=2$, the number of connected components of $S_{0}$ is exactly two, and the components of $F_{0}$ which intersects both $F_{1}$ and $F_{2}$ must be $S_{0,2}$. 
We call $A$ and $B$ for the two components of $F_{0}$, respectively. 
We have $S_{0}\approx S_{0, 2}\coprod S_{0, 2}$ or $S_{0}\approx S_{0, 2}\coprod S_{0, 3}$. 
At least one component of $F_{0}$ must intersect both $F_{1}$ and $F_{2}$, and without loss of generality we assume that $A$ intersects both $F_{1}$ and $F_{2}$, and so $A\approx S_{0,2}$. 

If $B$ intersects both $F_{1}$ and $F_{2}$, then $B\approx S_{0,2}$. 
In this case $F_{1}$ and $F_{2}$ need have at least two boundary components, and so $F_{1}, F_{2}\in\{N_{1,3}, S_{0,4}\}$.
We obtain Case (3) (see Figure~\ref{fig_case_3_4}).

If $B$ is disjoint from either $F_{1}$ or $F_{2}$ (without loss of generality we may assume that $B$ is disjoint from $F_2$), then $B\in\{S_{0,2}, S_{0,3}, N_{1,2}\}$, because $\alpha=2$.
In the case where $B\approx S_{0,2}$, both of the boundary components of $B$ should intersect $F_{1}$. 
Case (4) is immediate, because $F_{1}$ should have at least three boundary components (see Figure~\ref{fig_case_3_4}). 
In the case where $B\approx S_{0,3}$, it follows that $B\cap F_{1}\approx S^{1}$, because $\alpha=2$.
Hence Case (5) follows from the same reason as Case (4). 
In the case where $B\approx N_{1,2}$, by the same argument as Case (5) we obtain Case (6). 
\end{proof}

\begin{figure}[h]
\begin{center}
\includegraphics[scale=0.37]{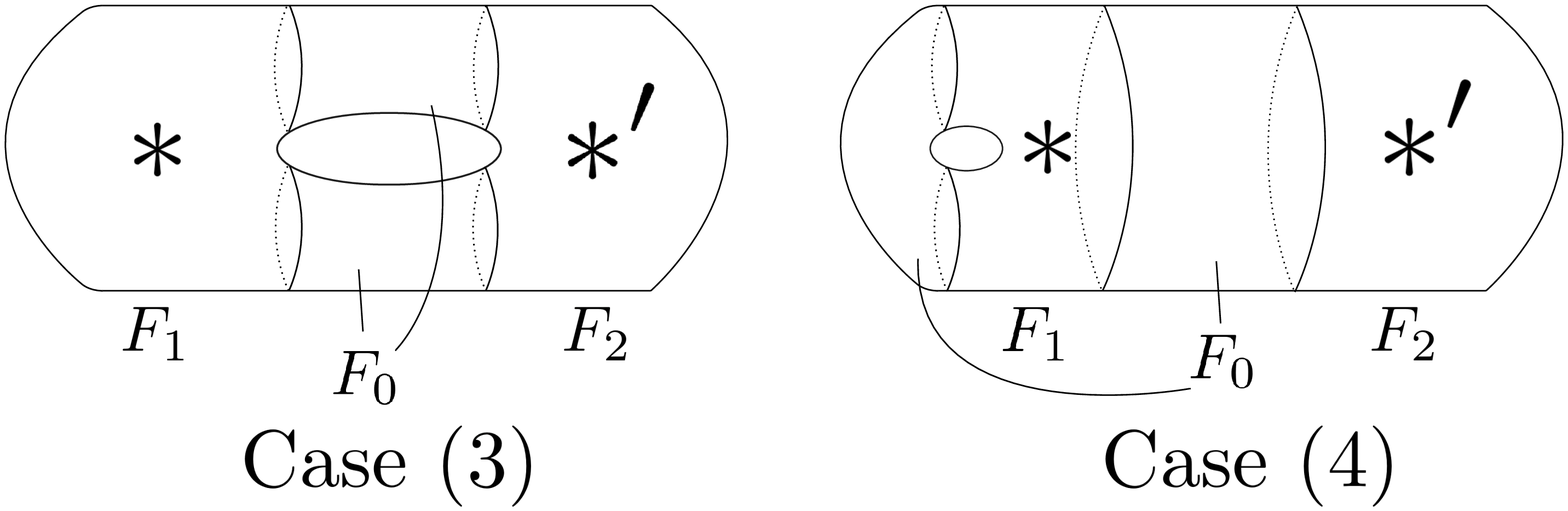}
\caption{The nonorientable surfaces of Cases (3) and (4), where $*$ and $*'$ are one of the cases of $F_{1}$ and $F_{2}$, respectively. \label{fig_case_3_4}}
\end{center}
\end{figure}

\begin{lemma}\label{impossibility_to_be_subgraph}
For $N=N_{1, 6},~N_{3, 3},~N_{5, 0}$, the graph $\Gamma_{0}$ is not embedded into $\mathcal{C}^{\mathrm{two}}(N)$ as a full subgraph.
\end{lemma}

\begin{proof}[Proof of Lemma~\ref{impossibility_to_be_subgraph}]
The strategy of the proof is to show that, for each case in Theorem~\ref{six_cases_of_the_two-sided_complexity_four_nonorientable_surfaces}, there is no systems of two-sided curves which realize 
$\Gamma_{0}$. 

We suppose that $\Gamma_{0}\leq \mathcal{C}^{\mathrm{two}}(N)$. 
We regard $a$, $b$, $c$, $d$, $g$, $h$, and $q$ as two-sided curves on $N$. 
From $C_{4}\leq \Gamma_{0}$, we have one of the six cases in Lemma~\ref{six_cases_of_the_two-sided_complexity_four_nonorientable_surfaces}. 
By the definition of two-sided curve graphs, each of the intersections $q\cap g$, $q\cap h$, $g\cap F_{2}$, $h\cap F_{1}$, and $g\cap h$ is not empty. 
Moreover $\partial q\subseteq F_{0}$ and $g\cap F_{1}=h\cap F_{2}=\emptyset$. 

In Cases (1), (3), (4), (5), and (6), $g$ and $h$ intersect on a component of $F_{0}$ which connects $F_{1}$ and $F_{2}$ and is homeomorphic to $S_{0, 2}$ since $g\cap F_{1}=\emptyset$ and $h\cap F_{2}=\emptyset$. 
This implies that $g\subseteq F_{2}$ and $h\subseteq F_{1}$, and so $g\cap h=\emptyset$. 
This is a contradiction. 

In Case (2), since $q\subseteq F_{0}\approx S_{0, 3}$ or $N_{1,2}$, we have $q=F_{0}\cap F_{1}$ or $q=F_{0}\cap F_{2}$ (two-sided curves on $S_{0,3}$ or $N_{1,2}$ cannot be essential). 
Without loss of generality, we may assume $q=F_{0}\cap F_{1}$. 
The curve $q$ is a separating curve which separates $F_{1}$ from $N$. 
By the fact that $g\cap F_{1}=\emptyset$, we have $g\cap q=\emptyset$. 
This contradicts the fact that $g\cap q\not=\emptyset$.

Therefore it follows that $\Gamma_{0}\not\leq \mathcal{C}^{\mathrm{two}}(N)$ with $\xi^{\mathrm{two}}(N)=4$. 
\end{proof}

Corollary~\ref{possibility_to_be_subgroup} together with Lemma~\ref{impossibility_to_be_subgraph} competes the proof of Theorem~\ref{opposite_is_not_true}. 

\section{Proofs of Theorems \ref{low_genus} and \ref{n-thick_stars} \label{final}}

In this section we give proofs for Theorems \ref{low_genus} and \ref{n-thick_stars}. 

Let $\iota \colon \mathrm{Mod}(N) \hookrightarrow \mathrm{Mod}(S)$ be an embedding given in Lemma \ref{ori_cov_mcg_inj}. 
According to Thurston~\cite{Thurston88}, the Nielsen-Thurston classification on $\mathrm{Mod}(S)$ can be extended to $\mathrm{Mod}(N)$. 
In fact Wu~\cite{Wu87} proved that a mapping class $h\in\mathrm{Mod}(N)$ is of {\it finite order} (resp.\ {\it reducible}, {\it pseudo-Anosov}) if $\iota(h)\in\mathrm{Mod}(S)$ is of {\it finite order} (resp.\ {\it reducible}, {\it pseudo-Anosov}). 

\begin{lemma} \label{anosov_prod_nonori}
Let $N$ be a nonorientable surface with negative Euler characteristic and $f$ be a pseudo-Anosov element of $N$. 
For every finite graph $\Gamma \leq \mathcal{C}^{two}(N)$, there exists a number $n$ such that the $n$-th powers of Dehn twists along the closed curves corresponding to $\Gamma$ and $f$ generate a subgroup isomorphic to $A(\Gamma) * \mathbb{Z}$ in ${\rm Mod}(N)$. 
\end{lemma}
\begin{proof}[Proof of Lemma \ref{anosov_prod_nonori}]
As in Section \ref{Proof_of_theorem_condition_raag_in_nonorimcg}, by $\tilde{\Gamma}$ we denote the subgraph of $\mathcal{C}(S)$ induced by a union $L$ of the sets of lifts of curves corresponding to $\Gamma$. 
Since $f$ is pseudo-Anosov in the orientation double cover $S$, Koberda's theorem \cite[Theorem 1.1]{Koberda12} implies that, for a sufficiently large $n$, the $n$-th powers of Dehn twists along curves in $L$ and $f$ induces an embedding $\psi \colon A(\tilde{\Gamma}) * \mathbb{Z} \hookrightarrow \mathrm{Mod}(S)$. 
As in the proof of Theorem \ref{condition_raag_in_nonorimcg}, we obtain an embedding $A(\Gamma) \hookrightarrow A(\tilde{\Gamma})$ so that the image is contained in $\mathrm{Mod}(N)$. 
Hence, by restricting $\psi$ to the subgroup generated by $A(\Gamma)$ and $f$, we have an embedding of $A(\Gamma) * \mathbb{Z}$ into $\mathrm{Mod}(N)$. 
\end{proof} 

The {\it support} of a homeomorphism $\varphi$ of $N$ is defined by 
\begin{align*}
{\rm supp}(\varphi)=\overline{\{p\in N\mid \varphi(p)\not=p\}}. 
\end{align*}

\begin{definition}\label{standard_embeding}
An injective homomorphism $f$ from $A(\Gamma)$ to ${\rm Mod}(N)$ is {\it standard} if $f$ satisfies the following two conditions. 
\begin{itemize}
\item[(i)] The map $f$ sends each generator $v$ of $A(\Gamma)$ to a non-trivial multi-twist, 
\item[(ii)] for two distinct generators $u$ and $v$ of $A(\Gamma)$, the support of $f(u)$ is not contained in the support of $f(v)$.
\end{itemize}
\end{definition}

The concept of standard embeddings of right-angled Artin groups into the mapping class groups of orientable surfaces was introduced by Kim--Koberda (see \cite[Lemma 2.3]{Kim--Koberda14} and \cite[Definition 10]{Kim-Koberda16}). 

\begin{remark}\label{remark_of_standard_embedding}
In Definition~\ref{standard_embeding}, if supp($f(v)$) induces a maximal clique in $\mathcal{C}^{\mathrm{two}}(N)$, then the condition (ii) implies that the link of $v$ is empty. 
\end{remark}

\begin{proof}[Proof of Theorem~\ref{low_genus}]
Assume that $\xi^{\mathrm{two}}(N)=0$, that is, $N=N_{1,0},~N_{1,1},~N_{1,2}$. 
In this case, the mapping class groups $\mathrm{Mod}(N)$ are all finite. 
Hence, if $A(\Gamma)\leq\mathrm{Mod}(N)$, then it follows that $A(\Gamma)=1$ because the subgroups of finite groups are finite and $A(\Gamma)$ is necessarily infinite excepting the trivial group. 
It means that $\Gamma=\emptyset$, and so $\Gamma\leq\mathcal{C}^{\mathrm{two}}(N)$. 

We next assume that $\xi^{\mathrm{two}}(N)=1$, that is, $N=N_{1,3},~N_{2,0},~N_{2,1},~N_{3,0}$. 
For $N=N_{2,0}$, the mapping class group $\mathrm{Mod}(N)$ is finite, and we are done. 
For $N=N_{1,3},~N_{2,1},~N_{3,0}$, any pair of two-sided curves on $N$ intersects non-trivially because of  the fact that $\xi^{\mathrm{two}}(N)=1$.
Therefore, the curve graphs $\mathcal{C}^{\mathrm{two}}$ consist of isolated vertices.
It is sufficient to prove that if $A(\Gamma)\leq\mathrm{Mod}(N)$, then $\Gamma$ has no edges. 
To see this, suppose that $\Gamma$ has an edge. 
Then we have $\mathbb{Z}^{2} \hookrightarrow A(\Gamma)$. 
However, since ${\rm Mod}(N)$ has no subgroup isomorphic to $\mathbb{Z}^{2}$, it holds that $A(\Gamma) \not \leq {\rm Mod}(N)$. 
So we are done. 

We finally assume that $\xi^{\mathrm{two}}(N)=2$, that is, $N=N_{1,4},~N_{2,2},~N_{3,1}$. 
Let $f$ be a standard embedding $f$ from $A(\Gamma)$ to ${\rm Mod}(N)$. 
Note that $\mathcal{C}^{\mathrm{two}}(N)$ is triangle-free. 
In addition, $\mathcal{C}^{\mathrm{two}}(N)$ has an infinite diameter, because $\mathcal{C}^{\mathrm{two}}(N)$ is quasi-isometric to the ordinal curve graph $\mathcal{C}(N)$ (see Masur-Schleimer~\cite[Section 6]{Masur-Schleimer13} for the reason) and $\mathcal{C}(N)$ has an infinite diameter by Bestvina-Fujiwara~\cite[Theorem 5.5]{Bestvina-Fujiwara07}. 
We claim that the conclusion of the theorem holds if and only if it holds for each connected component of $\Gamma$. 
This claim is an easy consequence of Lemma \ref{anosov_prod_nonori} and the fact that there exists a pseudo-Anosov homeomorphism on $N$. 
By the above claim, we may assume that $\Gamma$ is connected, and so it has at least one edge.
Each vertex $v$ of $A(\Gamma)$ is mapped to a power of a single Dehn twist $t_{c}$ along $c$ by $f$, because $\Gamma$ has no isolated vertex and $\mathcal{C}^{\mathrm{two}}(N)$ is triangle-free (see Remark~\ref{remark_of_standard_embedding}). 
Hence we gain a full embedding $\Gamma \rightarrow \mathcal{C}^{\mathrm{two}}(N)$. 
\end{proof}

\begin{proof}[Proof of Theorem~\ref{n-thick_stars}]
\[
\xymatrix@!C{
& \mathrm{Mod}(N) \ar@{-}[d]&\\
& A(\Gamma) \ar@{-}[ld]  \ar@{-}[rd] & \\
\mathbb{Z}^n \cong \langle A \cup C \rangle &  & \langle B \cup C \rangle \cong \mathbb{Z}^n \\
 \mathbb{Z}^n \cong f \langle K, v\rangle \ar@{-}[u]^{\text{finite-index}}& \langle C \rangle \ar@{-}[lu] \ar@{-}[ru]& f \langle L, v\rangle \cong \mathbb{Z}^n \ar@{-}[u]_{\text{finite-index}} \\
& f(\langle K, v\rangle \cap \langle L, v\rangle) \ar@{-}[lu] \ar@{-}[u] \ar@{-}[ru]& \\
& f\langle v \rangle \ar@{=}[u] \cong \mathbb{Z}&
}
\]

Let $A$ be a two-sided multi-curve on $N$. 
We denote by $\langle A\rangle$ a subgroup of ${\rm Mod}(N)$ which is generated by Dehn twists along the curves in $A$.
The proof of Theorem~\ref{n-thick_stars} is essentially the same as the proof of Kim-Koberda~\cite[Theorem 5]{Kim-Koberda16}.

Let $f$ be a standard embedding $f \colon A(\Gamma) \rightarrow \mathrm{Mod}(N)$. 
Pick a vertex $v$ in $V(\Gamma)$. 
We write $K$ and $L$ for two disjoint cliques of $\Gamma$ such that $K \cup \{v\}$ and $L \cup \{v\}$ are cliques on $n$ vertices of $\Gamma$. 
The support of $f \langle K\rangle$ is a regular neighborhood of a two-sided multi-curve in $N$, and we call $A$ for the two-sided multi-curve. 
Similarly, we write $B$ and $C$ for two-sided multi-curves in the supports of $f \langle L\rangle$ and $f \langle v\rangle$ in $N$, respectively. 
Since $\xi(N)=n$, two-sided multi-curves $A\cup C$ and $B\cup C$ are maximal. 
Note that $\langle C\rangle$ is a subgroup of $\langle A\cup C\rangle\cap\langle B\cup C\rangle$. 
The subgroup $f \langle v\rangle$ in $\mathrm{Mod}(N)$ is a finite index subgroup of $\langle C\rangle$. 
By the fact that $\langle C\rangle\cong \mathbb{Z}^{|C|}$ and $f \langle v\rangle\cong \mathbb{Z}$, it follows that $|C|=1$. 
Hence, for each $v\in V(\Gamma)$, $f(v)$ is some single Dehn twist $t_{a}$ along a two-sided curve $a$, and we obtain a full embedding of $\Gamma$ into $\mathcal{C}^{\mathrm{two}}(N)$. 
\end{proof}

\section{Appendix: an analogue of Lemma \ref{co_hom_inj} for right-angled Coxeter groups}

For a finite graph $\Gamma$, the {\it right-angled Coxeter group} $C(\Gamma)$ is defined to be the quotient of $A(\Gamma)$ by the normal closure $\langle \langle  v \in V(\Gamma) \rangle \rangle$. 
A {\it reduced word} in a right-angled Coxeter group can be defined similarly as a reduced word in a right-angled Artin group; a reduced word for an element $g$ in a right-angled Coxeter group $C(\Gamma)$ is a shortest word representing $g$ in $C(\Gamma)$. 
Tits' solution to the word problem for right-angled Coxeter groups is as follows.  

\begin{lemma}[{cf.\ \cite[Theorem 3.4.1]{Davis}}]
\label{reduced_RACG}
Let $w$ be a word in $C(\Gamma)$. 
Then $w$ is reduced if and only if $w$ contains neither a word of the form $v^{\varepsilon}xv^{- \varepsilon}$ nor a word of the form $v^{\varepsilon}xv^{\varepsilon}$, where $v$ is a vertex of $\Gamma$, $\varepsilon \in \{1, -1 \}$ and $x$ is a word such that $v$ is commutative with all of the letters in $x$. 
\end{lemma}

Note that the Tits gave a solution to the word problem for general Coxeter groups. 
We now establish an analogue of Lemma \ref{co_hom_inj} for right-angled Coxeter groups. 

\begin{proposition} \label{RACG}
Let $f \colon \Lambda \rightarrow \Gamma$ be a surjective full map with condition $(*)$ in Lemma \ref{co_hom_inj}. 
Suppose that every fiber of the map $f$ forms a clique. 
Then the diagonal map $\phi \colon V(\Gamma) \rightarrow C(\Lambda)$ sending each vertex $u$ to the product $\coprod_{v \in f^{-1}(u)} v$ extends to an injective group homomorphism $\varphi \colon C(\Gamma) \hookrightarrow C(\Lambda)$. 
Moreover, $\varphi$ is a quasi-isometric embedding. 
\end{proposition}
\begin{proof}[Proof of Proposition \ref{RACG}]
We first prove that $\phi$ extends to a group homomorphism $\varphi$. 
Pick a vertex $v$ of $\Gamma$. 
Since $f(v)$ consists of mutually adjacent vertices, $\phi(v)^2=1$ in $C(\Lambda)$. 
We can check the commutativity of the images of adjacent vertices as in the proof of Lemma \ref{co_hom_inj}. 
Hence, $\phi$ extends to a group homomorphism $\varphi$. 

We next prove that $\varphi$ is injective. 
To see this, we claim that $\varphi$ maps each reduced word in $C(\Gamma)$ to a reduced word in $C(\Lambda)$. 
Pick a non-trivial element $g$ of $C(\Gamma)$ and choose a reduced representation $w$ of $g$. 
Since $\varphi$ is a group homomorphism, we can regard $\varphi(w)$ as a word representation of $\varphi(g)$. 
Suppose, on the contrary, that $\varphi(w)$ is not reduced. 
Then according to Lemma \ref{reduced_RACG}, a cancelation must occur in $\varphi(w)$ and there exists a subword of the form $v^{\varepsilon} x v^{- \varepsilon}$ or $v^{\varepsilon} x v^{\varepsilon}$ such that $v$ is a vertex of $\Lambda$ and $x$ is a word in $C(\Lambda)$ consisting of vertices commutative with $v$, and $\varepsilon \in \{ 1, -1 \}$. 
The prefix and suffix of $v^{\varepsilon} x v^{- \varepsilon}$ come from the identical vertex $u_{v}$ in $\Gamma$, namely, $v \in f^{-1}(u_{v})$. 
This implies that $w$ contains a subword of the form $u_v^{\varepsilon} \bar{x} u_v^{- \varepsilon}$ or $u_v^{\varepsilon} \bar{x} u_v^{\varepsilon}$, where $\bar{x}$ is a word in $C(\Gamma)$ and $\varphi(\bar{x})$ is a subword of $x$. 
Since $\varphi(\bar{x})$ is a subword of $x$ and since $f$ satisfies the property $(*)$, every letter  contained in $\bar{x}$ is adjacent to $u_v$ in $\Lambda$. 
This implies that a cancelation $u_v^{\varepsilon} \bar{x} u_v^{- \varepsilon} \equiv \bar{x}$ or $u_v^{\varepsilon} \bar{x} u_v^{\varepsilon} \equiv \bar{x}$ occurs in $w$, and therefore $w$ is not reduced.  This is a contradiction, and hence $\varphi$ maps each reduced word to a reduced word. 
Thus, we obtain the result that $\varphi$ maps each non-trivial element to a non-trivial element as well as the result that $\varphi$ is a bi-Lipschitz map with a Lipschitz constant $\mathrm{max} \{ \# f^{-1}(v) \mid v \in V(\Lambda) \}$. 
\end{proof}

As seen in Remark \ref{kapovich_rem}, the complementary map of a surjective and locally surjective graph morphism is a surjective full map with condition $(*)$ each of whose fibers consists of mutually adjacent vertices. 
So we have the following. 

\begin{corollary}
Let $\Lambda^c \rightarrow \Gamma^c$ be a surjective and locally surjective graph morphism. 
Then $C(\Gamma)$ is quasi-isometrically embedded in $C(\Lambda)$. 
\end{corollary}

The concept of a ``planar emulator" is an example of surjective and locally surjective graph homomorphism; see \cite{Kapovich}. 

{\bf Acknowledgements:} 
The authors are grateful to Anthony Genevois, Sang-hyun Kim and Donggyun Seo for helpful comments on Section \ref{Proof_of_theorem_condition_raag_in_nonorimcg} and the appendix. 
The authors also thank to Naoto Shida for drawing figures. 
The first author was supported by JSPS KAKENHI, the grant number 20J1431, and the second author was supported by JST, ACT-X, the grant number JPMJAX200D, Japan, and partially supported by JSPS KAKENHI Grant-in-Aid for Young Scientists, Grant Number 21K13791.

\end{document}